\documentclass[10pt,leqno]{amsart}
\usepackage{graphicx}
\usepackage{indentfirst,csquotes}

\usepackage{paralist,fancyhdr,etoolbox}
\usepackage{amsmath, amsthm, amscd, amsfonts, amssymb, graphicx, color}
\usepackage[utf8]{inputenc}
\usepackage{gensymb}
\usepackage{bm}
\usepackage[most]{tcolorbox}
\usepackage{xcolor}
\usepackage[colorlinks,citecolor=blue]{hyperref}
\usepackage{orcidlink}
\usepackage{cleveref}

\newtheorem{theorem}{Theorem}[]

\newtheorem{proposition}[theorem]{Proposition}

\DeclareMathOperator{\Aut}{Aut}
\DeclareMathOperator{\GL}{GL}

\newcommand{\Z}{\mathbb{Z}}

\hypersetup{ colorlinks=true, linkcolor=black, filecolor=black, urlcolor=black }

\begin{document}
\title{A Group with Exactly One Noncommutator} 
\author{Omar Hatem}
\author{Daoud Siniora}
\date{\today}
\address{The American University in Cairo, P.O.Box 74, Cairo, Egypt}
\email{omarhatem2002@aucegypt.edu}
\email{daoud.siniora@aucegypt.edu}
\maketitle

\let\thefootnote\relax

\begin{abstract}
The question of whether there exists a finite group of order at least three in which every element except one is a commutator has remained unresolved in group theory. We address this open problem by developing an algorithmic approach that leverages several group theoretic properties of such groups. Specifically, we utilize a result of Frobenius and various necessary properties of such groups, combined with Holt and Plesken’s list of finite perfect groups, to systematically examine all finite groups up to a certain order for the desired property. The computational core of our work is implemented using the computer system GAP. We discover two nonisomorphic groups of order 368,640 that exhibit the desired property. Our investigation also establishes that this is the smallest order of such a group. This study provides a positive answer to Problem 17.76 in the Kourovka Notebook. In addition to the algorithmic framework, this paper provides a structural description of one of the two groups found. 
\end{abstract} 

\section{Introduction}

Ore proved that every element of the alternating group $A_n$ is a commutator for $n\geq 5$, and he conjectured the same holds for all nonabelian finite simple groups \cite{Ore1951}. Ore's conjecture was then proved by Liebeck, O'Brien, Shalev, and Tiep \cite{Liebeck2010}. Thus, every element of every finite nonabelian simple group is a commutator. MacHale asked if there exists a finite group of order at least 3 that has exactly one noncommutator. This is Problem 17.76 of the Kourovka Notebook \cite{khukhro2024unsolvedproblemsgrouptheory}. 

Fite showed that the smallest finite group where the commutator subgroup contains noncommutators has order 96, see \cite{Fite1902}. Macdonald showed that if $G$ is a group and if $|G:Z(G)|^2 <|G'|$, then there exists elements in $G'$ which are not commutators \cite{Macdonald1986}. Isaacs also showed that a certain wreath product of an abelian group with a nonabelian group has a commutator subgroup with noncommutators \cite{Isaacs1977}. More examples of groups  where the set of commutators is a proper subset of the commutator subgroup can be found in a survey by Kappe and Morse \cite{Kappe2005}. In this paper, we give a positive answer to MacHale's question. Using the computer algebra system GAP \cite{GAP4}, we found two nonisomorphic groups of order $2^{10}\cdot 360$ which have exactly one noncommutator. Moreover, this is the smallest order of a group with this property. In this paper, we describe the structure of one of them.

We now recall the definitions of the terms used in this paper. Let $G$ be a group and let $g,h\in G$. The \textit{commutator} of $g$ and $h$ is $[g,h]=ghg^{-1}h^{-1}$. An element $c$ of $G$ is a \textit{commutator} if there exist $g\in G$ and $h\in G$ such that $c=[g,h]$. An element of $G$ is a \textit{noncommutator} if it is not a commutator. The \textit{commutator subgroup} $G'$ of $G$ is the subgroup generated by all the commutators of $G$. It is the smallest normal subgroup of $G$ such that the quotient of the original group by this subgroup is abelian. A group is \textit{perfect} if it equals its own commutator subgroup. For instance, every nonabelian simple group is perfect.

We learned in November 2025, after this work was accepted for publication, that Skresanov \cite{skresanov2025finitegroupsexactlynoncommutator} has also recently solved this problem. He gives an infinite family of groups, each having precisely one noncommutator. The smallest group in his family has order 16,609,443,840. He asks for the smallest order of such a group: here we show that it is 368,640.


\section{Computational Search}
We search for a finite group which has exactly one noncommutator. The first such example is the cyclic group of order 2. Are there any others? We used GAP to search for such a group. We first state some key properties about such a group and its unique noncommutator element that speed up the search.
\begin{proposition}
Let $G$ be a finite group with $|G|\geq 3$ that has exactly one noncommutator element $u\in G$. Then the following hold.  
\begin{itemize}
            \item The order of $u$ is $2$.
            \item The element $u$ belongs to the center $Z(G)$ of $G$.
            \item $G$ is a perfect group.
            \item The element $u$ is the product of two commutators. Consequently, the commutator length of $G$ is $2$.
\end{itemize}
\label{four_props}
\end{proposition}
\begin{proof} Let $u$ be the unique noncommutator element of $G$.
\begin{itemize}
    \item The inverse of a commutator is a commutator, and so the inverse of a noncommutator is a noncommutator. Since $u$ is unique, it equals its inverse. Thus, $u^2 = 1$.
    \item Since conjugation by $g\in G$ is an automorphism of $G$, it follows that $gug^{-1}$ must be a noncommutator as well. Since $u$ is the only noncommutator it must be $gug^{-1}=u$ for any $g\in G$.  Therefore, for any $g\in G$, we have $gu=ug$, which means $u \in Z(G)$.
    \item Let $S$ be the set of all commutators of $G$, so $|S|+1=|G|$. Since $G'=\langle S \rangle\subseteq G$ we get that $|G'|\geq |S| = |G|-1\geq 2$ and either $G'=S$ or $G'=G$. But $|G'|$ divides $|G|$, and thus $|G|=|G'|$ implying that $G=G'$ showing that $G$ is perfect.
    \item Since $|G| > 2$, we can choose some nontrivial commutator $g\in G$. Then $u=(ug)g^{-1}$. Since $g$ is a commutator, we know $g^{-1}$ is a commutator as well. If $ug$ were a noncommutator, then $ug=u$ because $u$ is the only noncommutator, and so $g=1$ contradicting that $g$ is nontrivial. So $ug$ is a commutator too. Therefore, $u$ is a product of two commutators, and thus the commutator length of $G$ is $2$.
\end{itemize}
\end{proof}

The third point above is more general. Let $G$ be a group of order $n$ with exactly $k$ noncommutators; if $n-k>\frac{n}{2}$, then $G$ is perfect.  Proposition \ref{four_props} tells us that we need to search in the class of finite perfect groups. In GAP we used Holt and Plesken's list of finite perfect groups \cite{holtplesken}. We use the properties of Proposition \ref{four_props} together with Frobenius' characterization (see below) of commutator elements via character theory to conduct an efficient search.

\begin{theorem}[Frobenius \cite{Frobenius1896}]
Let $G$ be a finite group with identity element 1 and let $\operatorname{Irr}(G)$ be the set of all irreducible characters $\chi:G\to \mathbb{C}$ of $G$. Then $g\in G$ is a commutator if and only if 
\[
    \sum_{\chi \in \operatorname{Irr}(G)} \frac{\chi(g)}{\chi(1)} \neq 0\,.
    \label{comm_criterion}
\]
\end{theorem}

When a group satisfies the properties of Proposition \ref{four_props}, we compute its character table and use the above theorem to count the number of commutators in the group. In Figure \ref{fig:pseudocode}, we present pseudocode for our algorithm. The reader may also run the entire GAP code in a Juypter Notebook \cite{mycode}. Our algorithm detected two perfect groups of order $368640$ which have exactly one noncommutator. These are \verb|PG368640.31| and \verb|PG368640.33| in the GAP library of perfect groups. We learn from GAP that \verb|PG368640.31| is a semidirect product $\mathbb{Z}_2^{10} \rtimes_\varphi A_6$ for some group homomorphism $\varphi:A_6\to \Aut(\Z_2^{10})$ and \verb|PG368640.33| is a series of semidirect products and non-splitting extensions of $\Z_2$ and $A_6$.   
\begin{figure}
\begin{tcolorbox}[colback=white, colframe=black!60, title=Pseudocode for the search algorithm]
    \begin{verbatim}
    set m = 1
    while m > 0
        for group G in perfect groups of order m
                check if commutator length is 2
                check if center is of even order
            if both pass 
                Get CharacterTable tbl of G 
                Get the conjugacy classes of G
    
            set n = 0
            # n is the number of commutators
                    
            #Iterate through every cojugacy class
            for conjugacy class C do:
                set sum = 0
                
                for every irreducible character chi in tbl
                        sum = sum + (chi[C] / chi[1]);
                End loop;
    
                if sum != 0 then
                        n = n + |C|;
            End loop;
                    
                if n+1 = |G| (group found!). 
    Terminate.
    \end{verbatim} \label{alg_for_search}
\end{tcolorbox}
   \caption{Pseudocode for the algorithm.}
    \label{fig:pseudocode}
\end{figure}

\section{Description of the group $\Z_2^{10} \rtimes_\varphi A_6$}

We first recall the definition of the semidirect product of groups. Given groups $N$ and $H$ together with a group homomorphism $\phi:H\to \Aut(N)$, their semidirect product $N\rtimes_\phi H$ with respect to $\phi$ is the group whose underlying set is the cartesian product $N\times H$ and where for elements $(n,h)$ and $(n',h')$ in $N\times H$ their product is defined to be
\[(n,h)\,(n',h')=(n\phi_h(n')\,,\,hh')\]
where $\phi_h=\phi(h)$.

To describe the group structure of $\mathbb{Z}_2^{10} \rtimes_\varphi A_6$ found by the algorithm we need to present its group homomorphism $\varphi:A_6 \to \Aut(\Z_2^{10})$. Recall that the operation of the group $\Z_2=\{0,1\}$ is addition modulo 2, and $A_6$ is the alternating group of degree 6 under function composition. Observe that $\Z_2^{10}$ is an abelian group of exponent $2$, so we view it as a vector space over the field $\mathbb{F}_2$. We think of elements of $\Z_2^{10}$ as column vectors of $0$s and $1$s of length $10$. Moreover, every automorphism of $\Z_2^{10}$ is represented by an invertible $10\times 10$ boolean matrix. Let this be expressed by the group isomorphism $\psi:\Aut(\Z_2^{10})\to \GL_{10}(\mathbb{F}_2)$ which sends an automorphism to the matrix whose columns are the images of the standard basis of $\Z_2^{10}$ under that automorphism. 

Furthermore, to understand the group homomorphism $\varphi:A_6 \to \Aut(\Z_2^{10})$ we only need to know its action on a generating set of $A_6$ to get a complete description of $\varphi$. Given $\alpha \in A_6$, we let $M_\alpha=(\psi\circ \varphi)(\alpha)$ be the invertible $10\times 10$ boolean matrix that represents the automorphism $\varphi(\alpha)\in\Aut(\Z_2^{10})$. The two cycles  $\sigma = (1,2,3,4,5)$ and $\eta = (4,5,6)$ generate $A_6$. Using GAP, we compute their corresponding matrices:
\[   
    M_\sigma = 
    \begin{bmatrix}
    	1 & 1 & 1 & 1 & 1 & 0 & 0 & 0 & 0 & 0 \\
            0 & 0 & 1 & 1 & 1 & 0 & 0 & 0 & 0 & 0 \\
            1 & 0 & 1 & 0 & 1 & 0 & 0 & 0 & 0 & 0 \\
            0 & 1 & 1 & 1 & 1 & 0 & 0 & 0 & 0 & 0 \\
            0 & 1 & 0 & 1 & 1 & 0 & 0 & 0 & 0 & 0 \\
            0 & 0 & 0 & 0 & 0 & 0 & 1 & 0 & 1 & 1 \\
            0 & 0 & 0 & 0 & 0 & 0 & 1 & 0 & 0 & 0 \\
            0 & 0 & 0 & 0 & 0 & 1 & 0 & 1 & 1 & 0 \\
            0 & 0 & 0 & 0 & 0 & 0 & 0 & 1 & 0 & 0 \\
            0 & 0 & 0 & 0 & 0 & 0 & 0 & 0 & 1 & 0 \\
    \end{bmatrix}
     \text{ ~~~ and ~~~ }
    M_\eta = 
    \begin{bmatrix}
    	0 & 1 & 0 & 0 & 1 & 0 & 0 & 0 & 0 & 0 \\
            1 & 1 & 0 & 0 & 1 & 0 & 0 & 0 & 0 & 0 \\
            0 & 1 & 0 & 1 & 1 & 0 & 0 & 0 & 0 & 0 \\
            1 & 0 & 1 & 1 & 0 & 0 & 0 & 0 & 0 & 0 \\
            0 & 0 & 0 & 0 & 1 & 0 & 0 & 0 & 0 & 0 \\
            0 & 0 & 0 & 0 & 0 & 1 & 1 & 0 & 1 & 0 \\
            0 & 0 & 0 & 0 & 0 & 0 & 0 & 0 & 1 & 0 \\
            0 & 0 & 0 & 0 & 0 & 0 & 1 & 0 & 0 & 0 \\
            0 & 0 & 0 & 0 & 0 & 0 & 0 & 1 & 0 & 0 \\
            0 & 0 & 0 & 0 & 0 & 0 & 0 & 0 & 0 & 1 \\
    \end{bmatrix}
    \label{matrices_gens}
\]

Since $\sigma$ and $\eta$ generate $A_6$, we can thus compute $M_\alpha$ for $\alpha \in A_6$. We can now describe the group operation of $\Z_2^{10} \rtimes_\varphi A_6$. Choose $(s,\alpha)$ and $(t,\beta)$ in $\mathbb{Z}_2^{10} \rtimes_\varphi A_6$. Their product is $$( s,\ \alpha )\,( t,\ \beta ) \ =\ ( s+M_{\alpha } t\ ,\ \alpha \beta ).$$ 
It remains to get $M_\alpha$ to find the product. We can compute $M_\alpha$ using the matrices $M_\sigma$ and $M_\eta$ above. Towards this, we express $\alpha$ as a product of the generators $\sigma$ and $\eta$ of $A_6$ and their inverses. For example, when $\alpha=\sigma^{-2}\eta\sigma^3$, we get

$$M_\alpha=(\psi\circ \varphi)(\alpha)=(\psi\circ \varphi)(\sigma^{-2}\eta\sigma^3)=(M_\sigma)^{-2} M_\eta (M_\sigma)^3.$$

Moreover, the inverse of an element $(s,\alpha)$ in $\mathbb{Z}_2^{10} \rtimes_\varphi A_6$ is $(M_{\alpha^{-1}}s, \alpha^{-1})$. We now give a formula for the commutator in $\mathbb{Z}_2^{10} \rtimes_\varphi A_6$. The commutator is  
\begin{equation}
[( s ,\ \alpha )\ ,\ ( t,\ \beta )]=(s+M_\alpha t+M_{\alpha\beta\alpha^{-1}}s+M_{[\alpha,\beta]}t\ ,\ [\alpha,\beta]).
\end{equation}
This is shown below. 

\begin{align*}
[( s,\ \alpha )\ ,\ ( t,\ \beta )] &= (s,\ \alpha )( t,\ \beta )( s,\ \alpha )^{-1} ( t,\ \beta )^{-1}   \nonumber\\
 &= (s+M_\alpha t\ ,\ \alpha\beta )( M_{\alpha^{-1}}s\ ,\ \alpha ^{-1}) ( M_{\beta^{-1}}t\ ,\ \beta ^{-1})   \nonumber\\
 &= (s+M_\alpha t\ ,\ \alpha\beta )( M_{\alpha^{-1}}s+M_{\alpha^{-1}\beta^{-1}}t\ ,\ \alpha ^{-1}\beta ^{-1})   \nonumber\\
 &= (s+M_\alpha t+M_{\alpha\beta}(M_{\alpha^{-1}}s+M_{\alpha^{-1}\beta^{-1}}t)\ ,\ \alpha\beta \alpha ^{-1}\beta ^{-1})   \nonumber\\
 &= (s+M_\alpha t+M_{\alpha\beta\alpha^{-1}}s+M_{\alpha\beta\alpha^{-1}\beta^{-1}}t\ ,\ \alpha\beta \alpha ^{-1}\beta ^{-1})   \nonumber\\
 &= (s+M_\alpha t+M_{\alpha\beta\alpha^{-1}}s+M_{[\alpha,\beta]}t\ ,\ [\alpha,\beta]).   
\end{align*}

Next, we search for the unique noncommutator element $u$ of $\mathbb{Z}_2^{10} \rtimes_\varphi A_6$ which we know is in the center. Pick $(c, \gamma)$ in the center of $\Z_2^{10} \rtimes_\varphi A_6$. Clearly, $\gamma$ must be in the center of $A_6$, which is centerless, and thus, $\gamma$ is the identity permutation $\varepsilon$ of $A_6$. Furthermore, the vector $c$ is an eigenvector of $M_\alpha$ with eigenvalue 1 for all $\alpha \in A_6$. To see this, let $(t,\alpha)$ be an element in $\mathbb{Z}_2^{10} \rtimes_\varphi A_6$. Since $(c, \varepsilon)$ is central,
\begin{align*}
    (c, \varepsilon)(t, \alpha)&=(t, \alpha)(c, \varepsilon)\\
    (c+M_\varepsilon t\ , \ \varepsilon\alpha)&=(t+M_\alpha c\ , \ \alpha\varepsilon)\\
    (c+t\ , \ \alpha)&=(t+M_\alpha c\ , \ \alpha)
\end{align*}
Thus, $c=M_\alpha\, c$, and so $c$ is an eigenvector of $M_\alpha$ for every $\alpha\in A_6$. Observe that a common eigenvector of the generators $M_\sigma$ and $M_\eta$ is an eigenvector of every $M_\alpha$.  This characterization of central elements was used to determine the center of $\mathbb{Z}_2^{10} \rtimes_\varphi A_6$. Using GAP, we compute the intersection of the eigenspaces of $M_\sigma$ and $M_\eta$. The center of $\Z_2^{10}\rtimes_\varphi A_6$ has 4 elements and one of them is the unique noncommutator $u=(q, \varepsilon)$ where  
\[
    q = 
    \begin{bmatrix}
        1 & 0 & 1 & 0 & 1 & 1 & 1 & 1 & 1 & 1
    \end{bmatrix}.
\]

More information on $\mathbb{Z}_2^{10} \rtimes_\varphi A_6$ appears in \cite[Section 5.3.13]{holtplesken}. For an independent proof establishing that $u=(q,\varepsilon)$ is the unique noncommutator of $\Z_2^{10}\rtimes_\varphi A_6$, we suggest using \cite[Proposition 25]{Serre1977} where, using the method of ``little groups" of Wigner and Mackey, Serre describes the irreducible characters of a group that is a semidirect product of an abelian normal subgroup $A$ with a subgroup $H$. In our setting $A=\Z_2^{10}$ and $H=A_6$.

\bigskip

\section*{Acknowledgements}
The authors would like to thank Isabel Müller and Eamonn O'Brien for their fruitful suggestions. They would also like to thank the referee for helpful suggestions on the presentation of this article.

\bibliographystyle{plain}
\bibliography{sample}
\end{document}